\documentclass[11pt]{amsart}

\usepackage[left=1in, right=1in, top=1.2in, bottom=1.2in]{geometry}
\usepackage{microtype, mathtools, mathrsfs, enumerate, dsfont, esvect}
\usepackage{verbatim}
\usepackage[linktocpage]{hyperref}
\hypersetup{
    colorlinks=true,        
    linkcolor=red,          
    citecolor=blue,         
    filecolor=magenta,      
    urlcolor=cyan           
}

\usepackage{amssymb}
\usepackage{url}

\newtheorem{theorem}{Theorem}[section]

\newtheorem{lemma}[theorem]{Lemma}

\newtheorem{corollary}[theorem]{Corollary}

\theoremstyle{definition}

\theoremstyle{plain}
\numberwithin{equation}{theorem}

\theoremstyle{remark}

\newtheorem{remark}[theorem]{Remark}
\newtheorem{example}[theorem]{Example}

\newif\ifhascomments \hascommentstrue
\ifhascomments
  \newcommand{\dragos}[1]{{\color{red}[[\ensuremath{\bigstar\bigstar\bigstar} #1]]}}
  \newcommand{\matt}[1]{{\color{red}[[\ensuremath{\spadesuit\spadesuit\spadesuit} #1]]}}
\else
  \newcommand{\dragos}[1]{}
  \newcommand{\matt}[1]{}
\fi

\begin{document}

\title[Dixmier-Moeglin equivalence and Morita equivalence]{The Dixmier-Moeglin equivalence, Morita equivalence, and homeomorphism of spectra}

\author{Jason P. Bell}
\author{Xingting Wang}
\author{Daniel Yee}

\address{Department of Pure Mathematics\\
University of Waterloo\\
200 W University Ave.
Waterloo, ON N2L 3G1\\
Canada}
\email{jpbell@uwaterloo.ca}

\address{Department of Mathematics\\
 Howard University\\
 2400 Sixth St. NW\\
  Washington DC, 20059\\
  USA}
  \email{wangxingting84@gmail.com}

\address{Department of Mathematics \\
Bradley University\\
1501 W Bradley Ave.\\
Peoria, IL 61625\\
USA}
\email{dyee@bradley.edu}

\begin{abstract}
Let $k$ be a field and let $R$ be a left noetherian $k$-algebra.  The algebra $R$ satisfies the Dixmier-Moeglin equivalence if the annihilators of irreducible representations are precisely those prime ideals that are locally closed in the ${\rm Spec}(R)$ and if, moreover, these prime ideals are precisely those whose extended centres are algebraic extensions of the base field.  We show that if $R$ and $S$ are two left noetherian $k$-algebras with ${\rm dim}_k(R), {\rm dim}_k(S)<|k|$ then if $R$ and $S$ have homeomorphic spectra then $R$ satisfies the Dixmier-Moeglin equivalence if and only if $S$ does.  In particular, the topology of ${\rm Spec}(R)$ can detect the Dixmier-Moeglin equivalence in this case.  In addition, we show that if $k$ is uncountable and $R$ is affine noetherian and its prime spectrum is a disjoint union of subspaces that are each homeomorphic to the spectrum of an affine commutative ring then $R$ satisfies the Dixmier-Moeglin equivalence.  We show that neither of these results need hold if $k$ is countable and $R$ is infinite-dimensional.  Finally, we make the remark that satisfying the Dixmier-Moeglin equivalence is a Morita invariant and finally we show that $R$ and $S$ are left noetherian $k$-algebras that satisfy the Dixmier-Moeglin equivalence then $R\otimes_k S$ does too, provided it is left noetherian and satisfies the Nullstellensatz; and we show that $eRe$ also satisfies the Dixmier-Moeglin equivalence, where $e$ is a nonzero idempotent of $R$.

\end{abstract}



\thanks{The first author was supported by a Discovery Grant from the National Sciences and Engineering Research Council of Canada.}

\maketitle

\section{Introduction}
Given a ring $R$, one of the most valuable methods of gaining information about the structure of $R$ is via its representation theory; that is, by first understanding the structure of the simple left $R$-modules and then using this information to gain insight into the ring itself.  Although this is a highly useful method for many classes of rings, in practice it is often very difficult to do so, and so often one instead settles for a coarser understanding of the representation theory by instead understanding the annihilators of simple modules; that is the primitive ideals of $R$.  If the Jacobson radical of $R$ is zero then $R$ is a subdirect product of rings $R/P$, where $P$ ranges over the primitive ideals; and by Jacobson's density theorem, rings of the form $R/P$ are dense subrings of rings of linear operators, and so many structure-theoretic problems about a ring can still be resolved with a sufficiently good understanding of the primitive ideals of a ring. 

One of the most beautiful results in this direction of characterizing primitive ideals is the work of Dixmier and Moeglin \cite{Dix77, Moe80}, who showed that if $L$ is a finite-dimensional complex Lie algebra then the primitive ideals of the enveloping algebra $U(L)$ are precisely the prime ideals of ${\rm Spec}(U(L))$ that are locally closed in the Zariski topology.  In addition to this, they proved that a prime ideal $P$ of $U(L)$ is primitive if and only if the Goldie ring of quotients of $U(L)/P$ has the property that its centre is just the base field of the complex numbers.  In general, for a field $k$ and a left noetherian $k$-algebra $R$, prime ideals $P$ for which the centre of the Goldie ring of quotients $Q(R/P)$ of $R/P$ has the property that its centre is an algebraic extension of $k$ are called \emph{rational} prime ideals.  Hence Dixmier and Moeglin's result can be regarded as saying that for primes $P$ of ${\rm Spec}(U(L))$ we have the following equivalences:
$$P ~{\bf locally~closed}\iff P~{\bf primitive}\iff P~{\bf rational}.$$
 In their honour, we today say that a left noetherian $k$-algebra $R$ satisfies the \emph{Dixmier-Moeglin equivalence} if we have the equivalence of the three above properties for primes in the spectrum of $R$.  It is now known that the Dixmier-Moeglin equivalence is a very general phenomenon that holds for many classes of algebras beyond just enveloping algebras of finite-dimensional Lie algebras.  Some additional examples include affine PI algebras \cite[2.6]{Von}, group algebras of nilpotent-by-finite groups \cite{Z}, various quantum algebras \cite{GoLet} (and see \cite[II.8.5]{BrGo}), affine cocommutative Hopf algebras of finite Gelfand-Kirillov dimension in characteristic zero \cite{BL14}, Hopf Ore extensions of affine commutative Hopf algebras \cite{BSM18}, twisted homogeneous coordinate rings of surfaces \cite{BRS10}, Hopf algebras of Gelfand-Kirillov dimension two (under mild homological assumptions) \cite{GZ}, and even in settings where the noetherian property does not hold and in which one must suitably modify the rationality property \cite{ABR, Lor09, Lor08}.
 Nevertheless, the equivalence is not universal and there are several finitely generated noetherian counterexamples are now known \cite{BLLM17, BCM17, Irv79, Lor}.  
 
 In this short note, our main result is to show that for a left noetherian $k$-algebra $R$ with the property that ${\rm dim}_k(R)<|k|$, the poset of prime ideals can detect whether the Dixmier-Moeglin equivalence holds.

\begin{theorem} 
\label{thm:main1}
Let $k$ be a field and let $R$ and $S$ be left noetherian $k$-algebras with ${\rm dim}_k(R),{\rm dim}_k(S)<|k|$ and suppose that there is an inclusion-preserving bijection between the poset of prime ideals of $R$ and the poset of prime ideals of $S$.  Then $R$ satisfies the Dixmier-Moeglin equivalence if and only if $S$ satisfies the Dixmier-Moeglin equivalence.  In particular, if ${\rm Spec}(R)$ and ${\rm Spec}(S)$ are homeomorphic then $R$ satisfies the Dixmier-Moeglin equivalence if and only if $S$ does.
\end{theorem}

This result is somewhat curious, because it says that for sufficiently large base fields the underlying topology on the prime spectrum ``sees'' the Dixmier-Moeglin equivalence. This is perhaps somewhat surprising because the rationality property has no obvious strong connection to the topological structure of the prime spectrum of a ring.  In fact, we are able to give an example, inspired by a construction of Lorenz \cite{Lor}, to show that if the hypothesis ${\rm dim}_k(R)<|k|$ does not hold then one can have an example of algebras with homeomorphic spectra in which the Dixmier-Moeglin equivalence holds for one but not the other.

We are able to prove a related theorem.
\begin{theorem}\label{thm:main2} Let $k$ be an uncountable field and let $R$ be a finitely generated left noetherian $k$-algebra.  Suppose that ${\rm Spec}(R)$ is a finite disjoint union of locally closed subsets $X_1,\ldots ,X_d$ with each $X_i$, when endowed with the subspace topology, homeomorphic to the prime spectrum of an affine commutative $k$-algebra.  Then $R$ satisfies the Dixmier-Moeglin equivalence.
\end{theorem}
The relevance of this theorem is seen in the fact that many quantum algebras have prime spectra of this form \cite{GoLet}, although generally in these cases work of Goodearl and Letzter \cite{GoLet} allows one to deduce that the Dixmier-Moeglin equivalence holds.  The stratification is a key ingredient in the important work of Goodearl and Letzter in obtaining the Dixmier-Moeglin equivalence for many classes of quantum algebras, where to obtain the Dixmier-Moeglin equivalence they use additional information about the stratification coming from the rational action of an affine algebraic group on the algebra.  Theorem \ref{thm:main2} again shows that for large base fields having an abstract stratification with parts homeomorphic to affine schemes of finite type over $k$ immediately gives the Dixmier-Moeglin equivalence without having any underlying action of an algebraic group. Again, the counterexample we produce to Theorem \ref{thm:main1} when $k$ is countable is a finitely generated noetherian algebra whose prime spectrum is homeomorphic to the spectrum of a polynomial ring in one variable over $k$ and hence Theorem \ref{thm:main2} does not hold if one relaxes the hypothesis that $k$ be uncountable.

For the remainder of the paper we consider the Dixmier-Moeglin equivalence in three settings where there is known to be a strong relationship between prime spectra; namely, Morita equivalence (where equivalent rings have homeomorphic spectra), rings of the form $eRe$ with $e$ an idempotent, in which there is a relationship between ${\rm Spec}(eRe)$ and an open subset of ${\rm Spec}(R)$, and tensor products of rings (in which the spectrum shares a strong relationship with the cartesian product of the spectra of the underlying rings).  The first setting we consider is Morita equivalence of algebras, which is a somewhat stronger property than having homeomorphic spectra. Two rings are Morita equivalent if they have equivalent categories of left modules (which in turn gives that the categories of right modules over these rings are equivalent).  A ring theoretic property is called a Morita invariant if whenever a ring has this property then every ring that is Morita equivalent to it also has this property. Many ring theoretic properties are well known to be Morita invariants, including being Artinian, noetherian, prime, semiprime, and having finite left or right global dimension \cite[Proposition 3.5.10]{McR}.  We show that satisfying the Dixmier-Moeglin equivalence is a Morita invariant.  Much of this is already known and we stress that this is more of an observation, since it is well known that there is an inclusion-preserving bijection between the prime spectra of Morita equivalent rings that preserves primitivity and hence it is not much additional work to obtain that the Dixmier-Moeglin equivalence is a Morita invariant.  A coarser notion of equivalence is derived equivalence and we do not know whether satisfying the Dixmier-Moeglin equivalence is a derived invariant.

In addition to results concerning Morita equivalent rings, we prove that if $R$ satisfies the Dixmier-Moeglin equivalence and $e$ is a nonzero idempotent then $eRe$ satisfies the Dixmier-Moeglin equivalence---in this case it is well known that there is a continuous surjection from an open subset of ${\rm Spec}(R)$ onto ${\rm Spec}(eRe)$ (see Remark \ref{rem:spec}) and we give an application to invariant subalgebras of the form $A^H$, where $H$ is a finite-dimensional semisimple Hopf algebra that acts on an algebra $A$.  Finally, we show that if $R$ and $S$ are left noetherian $k$-algebras that both satisfy the Dixmier-Moeglin equivalence then so does $R\otimes_k S$ under a hypothesis on the cardinality of the base field.  That is, the Dixmier-Moeglin equivalence is closed under the process of taking tensor products of reasonably well behaved algebras.  We now make these statements precise.
\begin{theorem}\label{thm:main3}
Let $k$ be a field and let $R$ and $S$ be left noetherian $k$-algebras.  Then we have the following:
\begin{enumerate}
\item[(a)] If $R$ and $S$ are Morita equivalent then $R$ satisfies the Dixmier-Moeglin equivalence if and only if $S$ does;
\item[(b)] If $R$ satisfies the Dixmier-Moeglin equivalence and $e$ is a nonzero idempotent then $eRe$ satisfies the Dixmier-Moeglin equivalence;
\item[(c)] If $R$ and $S$ satisfy the Dixmier-Moeglin equivalence and if $R\otimes_k S$ is left noetherian and satisfies the Nullstellensatz then $R\otimes_k S$ satisfies the Dixmier-Moeglin equivalence.
\end{enumerate}
\end{theorem}

The outline of the paper is as follows.  In \S2 we prove Theorems \ref{thm:main1} and \ref{thm:main2} and give an example to show that the conclusion to the statements of these theorems does not hold if we remove the hypotheses on the cardinality of the base field.  In \S3 we prove that satisfying the Dixmier-Moeglin equivalence is a Morita invariant and prove Theorem \ref{thm:main3}(b) and give an application of this result to invariant subalgebras.  Finally in \S4 we prove Theorem \ref{thm:main3}(c). 

\section{Homeomorphic spectra and stratification}
In this section, we prove Theorems \ref{thm:main1} and \ref{thm:main2} and show they do not hold without the hypotheses on the cardinality of the base field. To do this, we need a few preliminary results.  The following lemma is closely related to results that are well known---specifically, those concerning ideals in rings obtained by extending scalars in centrally closed algebras.  This lemma is somewhat stronger than these results but requires a hypothesis that when one extends scalars one obtains a prime ring in order to work.
\begin{lemma} Let $R$ be a prime noetherian $k$-algebra and suppose that $(0)$ is rational.  If $F$ is an extension of $k$ such that $Q(R)\otimes_k F$ is a prime ring then every nonzero ideal of $R\otimes_k F$ contains an element of the form $r\otimes 1$ with $r\neq 0$.\label{lem1}
\end{lemma}
\begin{proof} Let $I$ be a nonzero ideal of $R\otimes_k F$ and pick a nonzero element $x\in I$ with $x=\sum_{i=1}^d a_i\otimes \lambda_i$ with $d$ minimal.  We claim that $d=1$.  To see this, suppose that $d>1$.  We note that any element of the form 
$$\sum_{j=1}^m (c_j\otimes 1) x(d_j\otimes 1) = \sum_{i=1}^d \left(\sum_j c_j a_i d_j\right)\otimes \lambda_i $$ is again in $I$.  By minimality of $d$, $a_1$ is nonzero; and since the two-sided ideal generated by $a_1$ contains a regular element, by the above remarks, we may assume that $a_1$ is regular. 
Then for $r\in R$ we have
$$x(ra_1\otimes 1) - (a_1r\otimes 1)x = \sum_{i=2}^d (a_i r a_1 - a_1 r a_i)\otimes \lambda_i .$$  By minimality of $d$ we have $\lambda_2,\ldots ,\lambda_d$ are linearly independent over $k$ and hence by minimality of $d$ we then have
$a_i r a_1 - a_1 r a_i=0$ for all $r\in R$.  In particular, taking $r=1$ we see that $[a_i,a_1]=0$ for all $i$ and so
$a_1^{-1} a_i $ commutes with every $r\in R$.  Then $a_1^{-1}a_i\in Z(Q(R))$ and since $(0)$ is rational we have that 
$a_1^{-1} a_i$ is algebraic over $k$.
Then in $Q(R)\otimes_k F$ we may write $x=(a_1\otimes 1)\left(\sum_{i=1}^d z_i\otimes \lambda_i\right)$ with $z_1,\ldots ,z_d\in Z(Q(R))$.   Let $Z_0$ denote the finite extension of $k$ generated by $z_1,\ldots ,z_d$.  Then we have
$x = (a_1\otimes 1) y$ for some nonzero $y\in Z_0\otimes_k F$.  Since $Q(R)\otimes_k F$ is prime, we see that $Z(Q(R))\otimes_k F$ is an integral domain and since $[Z_0:k]<\infty$, we see that $y$ is algebraic over $F=k\otimes_k F$.  In particular, there is a non-trivial relation $y^m (1\otimes c_m)+ y^{m-1} (1\otimes c_{m-1}) + \cdots + (1\otimes c_0)=0$ with the $c_i\in F$.  Furthermore, we may assume $c_0$ is nonzero since $Z_0\otimes_k F$ is an integral domain.  In particular, we may assume $c_0=1$.  Then by construction
$$\sum_{j=0}^m (a_1\otimes 1)^{m-j} x^j (1\otimes c_j) =0$$ and since
$$\sum_{j=1}^m (a_1\otimes 1)^{m-j} x^j (1\otimes c_j)\in I,$$ we then see that
$a_1^m \otimes 1\in I$.  Since $a_1$ is regular, we have $a_1^m$ is nonzero and the result follows.
\end{proof}
\begin{remark} We note that this result need not hold if $Q(R)\otimes_k F$ is not a prime ring.  For example, if $R=F=\mathbb{Q}(\sqrt{2})$ and $k=\mathbb{Q}$.  Then $R\otimes_k F$ is not an integral domain and hence there is some nonzero prime ideal $P$ of $R\otimes_k F$.  But $R$ is a field, so $P\cap (R\otimes 1)$ is necessarily zero.
\end{remark}

The next result is the key part of the proof of Theorem \ref{thm:main1}: it shows that rationality of prime ideals can be captured purely in terms of the poset of prime ideals when the base field is sufficiently large compared to the dimension of the algebra.
\begin{lemma} Let $k$ be a field and let $R$ be a prime Noetherian $k$-algebra and suppose that ${\rm dim}_k(R)<|k|$. Then the zero ideal is rational in $R$ if and only if there is a set $X$ of cardinality less than $|k|$ and 
a set of nonzero prime ideals $\{P_x\colon x\in X\}$ such that every nonzero prime ideal $P$ of $R$ contains $P_x$ for some $x\in X$.
\label{card}
\end{lemma}
\begin{proof} Notice that if $|k|\le \aleph_0$ then ${\rm dim}_k(R)<\infty$ and so $R$ is a prime Artinian $k$-algebra and hence $R$ is simple.  In this case $(0)$ is a rational, maximal ideal and the claim is vacuously true in this case.  Thus we assume henceforth that $k$ is uncountable.

First suppose that $(0)$ is not rational.  Then there is some $z=ab^{-1}\in Z(Q(R))$ that is not algebraic over $k$, where $a,b\in R$ are regular elements.  Let $\mathcal{B}=\{r_{\alpha}\colon \alpha\in J\}$ be a basis for $R$ where $J$ is an index set with $|J|<|k|$.  Since $z$ is not in $k$, we have $a$ and $b$ are linearly independent over $k$ and we may assume without loss of generality that $a,b\in \mathcal{B}$.

Then 
for $\alpha,\beta\in J$, we have
$r_{\alpha}r_{\beta}=\sum_{\gamma\in J} c_{\alpha,\beta}^{(\gamma)} r_{\gamma}$, where 
$c_{\alpha,\beta}^{(\gamma)}$ is zero for all but finitely many $\gamma\in J$.  Let 
$T=\{c_{\alpha,\beta}^{(\gamma)}\colon \alpha,\beta,\gamma\in J\}$.  Then notice that 
since for $(\alpha,\beta)\in J\times J$ there are only finitely many values of $\gamma$ for which $c_{\alpha,\beta}^{(\gamma)}$ is nonzero we have injection from
$T\to J\times J\times \mathbb{N}$.  In particular, $|T|\le |J|^2 |\mathbb{N}| < |k|$, since $k$ is uncountable and $|J|<|k|$.  

Now let $k_0$ denote the prime subfield of $k$ and let $F$ denote the extension of $k_0$ generated by $T$.
Then the cardinality of $F$ is at most $\max(|T|,\aleph_0)$.  

Notice that if we let $R_0$ denote the $F$-subalgebra of $R$ generated by $\{r_{\alpha}\colon \alpha\in J\}$ then by construction $R_0 = V:=\sum_{\alpha\in J} F r_{\alpha}$ as an $F$-vector space since by construction any product of elements in $V$ is again in $V$ since $F$ contains $T$.  Furthermore, by construction we have $R_0\otimes_F k\cong R$.  Notice that $R_0$ is prime since $R$ is prime and since $R$ is a free $R_0$-module, we have that there is an inclusion-preserving injection from the set of left $R_0$-modules to the set of left $R$-modules by extension.  In particular, $R_0$ is left noetherian since $R$ is.  Then $Q(R_0)\otimes_F k$ is a localization of $R$ and since $a,b\in R_0$ we have $z\in Q(R_0)$.  Moreover, since ${\rm dim}_F (R_0) < |k|$ and since $|F|<|k|$ and since $k$ is uncountable, we have $|R_0|<|k|$.  Moreover, since every element of $Q(R_0)$ can be expressed in the form $sr^{-1}$ with $s,r\in R_0$ we have that $|Q(R_0)|<|k|$.   
Now for $\lambda\in k$, let $z_{\lambda}  = z\otimes 1 - 1\otimes \lambda$.  We note that by the Amitsur trick that the set of $\lambda\in k$ for which $z_{\lambda}$ is a unit in $Q(R)\otimes_F k$ must have cardinality strictly less than $|k|$.  Explicitly, since $Q(R_0)\otimes_F k$ has dimension strictly less than $|k|$ if 
$$Y:=\{\lambda\in k\colon z_{\lambda}~{\rm is~a~unit~in~}Q(R)\otimes_F k\}$$ has cardinality $|k|$ then there is necessarily a (finite) $k$-dependence of $z_{\lambda}^{-1}$ with $\lambda\in Y$; after clearing denominators in this dependence, we get that $z$ is algebraic over $k$, which is a contradiction.  

Since $R$ is left noetherian, we see that $Q(R_0)\otimes_F k$ is also left noetherian, as it is a localization of $R$.
It follows from Jategaonkar's principal ideal theorem \cite{Jat} that for $\lambda\in k\setminus Y$ we have $z_{\lambda} Q(R)\otimes_F k$ is contained in a height one prime ideal $P_{\lambda}$.  Then we let 
$Q_{\lambda}=P_{\lambda}\cap R$, which is a height one prime ideal of $R$ for $\lambda\in k\setminus Y$.  Now suppose that we have a set $X$ with $|X|<|k|$ and a set of nonzero prime ideals $\{J_x\colon x\in X\}$ such that every nonzero prime ideal contains some some $J_x$.  Then since $|X|<|k|$ and $|k\setminus Y|=|k|$ we then have that there is some $x\in X$ such that $J_x\subseteq Q_{\lambda}$ for infinitely many $\lambda\in k\setminus Y$.  But the $Q_{\lambda}$ are height one primes of $R$ and hence $L:=\bigcap Q_{\lambda} = (0)$, since by Noether's theorem there is a finite set of primes that are minimal with respect to containing $L$ and if $L$ is nonzero then since each $Q_{\lambda}$ is height one it must be minimal over $L$.  In particular, we get a contradiction.  Thus we have shown one direction.

For the remaining direction, suppose that $(0)$ is rational.  Then as above we have that $R\cong R_0\otimes_F k$ where $k$ is an extension of $F$, $R_0$ is a prime left noetherian $F$-algebra and $|F|,|Q(R_0)|<|k|$.  By Lemma \ref{lem1}, since $(0)$ is rational and $Q(R_0)\otimes_F k$ is a prime ring, every nonzero ideal of $R$ contains an element of the form $a\otimes 1$ with $a$ a nonzero element of $R_0$.  For nonzero $a\in R_0$, we let
$X_a$ denote the set of prime ideals of $R$ that contain $a\otimes 1$.  Then the nonzero prime ideals of $R$ is equal to the union of the $X_a$ as $a$ ranges over nonzero elements of $R_0$.  Now for each nonzero $a\in R_0$ let 
$J_a = \bigcap_{Q\in X_a} Q$.  Then $J_a$ is a nonzero semiprime ideal of $R$ since $a\otimes 1\in J_a$.  Thus by Noether's theorem, $J_a$ is a finite intersection of prime ideals $Q_{a,1}\cap \cdots \cap Q_{a,n_a}$ and every nonzero prime ideal containing $J_a$ contains $Q_{a,i}$ for some $i$.  We now let
$U=\{Q_{a,i}\colon a\in R_0\setminus 0, i\in \{1,\ldots ,n_a\}\}$.  Then $U$ is a set of nonzero prime ideals of $R$ of cardinality at most $|R_0|\times |\mathbb{N}| < |k|$ and by construction, every nonzero prime ideal of $R$ contains some prime ideal from $U$.  This completes the proof.
\end{proof}

We are now able to give the proof of our main result.

\begin{proof}[Proof of Theorem \ref{thm:main1}] We may assume that $R$ satisfies the Dixmier-Moeglin 
equivalence. Since ${\rm dim}_k(S)<|k|$, we have that $S$ satisfies the Nullstellensatz \cite[II.7.16]{BrGo}.  
Hence by a result from the book of Brown and Goodearl \cite[II.7.15]{BrGo}, we have the implications $$P~{\rm locally ~closed}\implies P~{\rm primitive}
\implies P~{\rm rational}$$
for $P\in {\rm Spec}(S)$.  Hence it suffices to prove that a rational prime ideal of $S$ is locally closed.  
Fix an inclusion preserving bijection $\Psi: {\rm Spec}(R)\to {\rm Spec}(S)$ and let $Q=\Psi^{-1}(P)$.  Then the prime ideals containing $Q$ are precisely the prime ideals $\{\Psi^{-1}(J)\colon J\supsetneq P\}$.  Since $P$ is rational, we have by Lemma \ref{card} that there is a set $X$ with $|X|<|k|$ and a set of prime ideals $\{J_x\colon x\in X\}$ of prime ideals of $S$ that properly contain $P$ such that every prime ideal of $S$ that properly contains $P$ must contain $J_x$.  Then since $\Psi$ is an inclusion-preserving bijection we then see that every prime ideal that properly contains $Q$ must contain a prime ideal from $\{\Psi^{-1}(J_x)\colon x\in X\}$.  In particular, by Lemma \ref{card} we see that $Q$ is a rational prime ideal of $R$.  Then since $R$ satisfies the Dixmier-Moeglin equivalence, we have that $Q$ is locally closed.  Thus by Noether's theorem there is a finite set of prime ideals $Q_1,\ldots ,Q_d$ of $R$ that properly contain $Q$ such that every prime ideal that properly contains $Q$ must contain some $Q_i$.  But then every prime ideal that properly contains $P$ must contain $\Psi(Q_i)$ for some $i\in \{1,\ldots ,d\}$.  But this means that $P$ is locally closed in ${\rm Spec}(S)$.  This completes the proof.
\end{proof}
We are able to prove Theorem \ref{thm:main2} now.
\begin{proof}[Proof of Theorem \ref{thm:main2}] Again since ${\rm dim}_k(R) <|k|$ we have that $R$ satisfies the Nullstellensatz and so it suffices to prove that a rational prime ideal is locally closed in ${\rm Spec}(R)$.  Let $P$ be rational.  Then by Lemma \ref{card}, there is a countable set (possibly finite or empty) of prime ideals $Q_1,Q_2,\ldots $ that properly contain $P$ such that every prime ideal properly containing $P$ contains $Q_i$ for some $i$.  If $\{Q_i\colon i\ge 1\}$ is finite then $P$ is locally closed and we are done. Thus we may assume that the set $\{Q_i\colon i\ge 1\}$ is infinite.  We have ${\rm Spec}(R)=X_1\sqcup X_2\sqcup\cdots\sqcup X_d$ with each $X_i$ homeomorphic to ${\rm Spec}(T_i)$ for some finitely generated commutative $k$-algebra $T_i$.  
For each $j\in \{1,\ldots ,d\}$ we let
$$L_j:=\bigcap_{\{i \colon Q_i\in X_j\}} Q_i.$$ Then for $j\in \{1,\ldots ,d\}$ we either have $L_j=P$ or $L_j$ properly contains $P$ in which case there is a finite set $S_j$ of primes that are minimal with respect to containing $L_j$.  If $L_j$ properly contains $P$ for $j=1,\ldots ,d$ then every prime ideal properly containing $P$ contains some prime ideal from the finite set $\bigcup_{j=1}^d S_j$ and hence $P$ is locally closed and we are done.  Alternatively, we have $L_j=P$ for some $j$.  In this case, $P\in X_j$ since $X_j$ is locally closed and we also have that $P$ is not locally closed in $X_j$ since $L_j=P$.  Thus since $T_j$ satisfies the Dixmier-Moeglin equivalence, we see by Lemma \ref{card} that there are uncountably many height one primes in $T_j/P$ (where we now identify $P$ with its image in ${\rm Spec}(T_j)\cong X_j$) and in particular, there are uncountably many height one prime ideals in $R/P$ and so $P$ is not a rational prime ideal of $R$ by Lemma \ref{card}, a contradiction.  
\end{proof}
\begin{remark}
We note that in the proof of Theorem \ref{thm:main2} the only place that $T_i$ is affine commutative is used is to ensure that the $T_i$ satisfy the Dixmier-Moeglin equivalence and so one can relax the statement of the theorem to only require that the $T_i$ be $k$-affine algebras that satisfy the Dixmier-Moeglin equivalence. 
\end{remark}

We now show that the conclusions to the statements of Theorems \ref{thm:main1} and \ref{thm:main2} do not hold if we relax the hypotheses on the cardinality of the base field. We note that Lorenz \cite{Lor} gives an example of an algebra that does not satisfy the Dixmier-Moeglin equivalence.  We are unable to use his example to produce a counterexample to the statements of Theorems \ref{thm:main1} and \ref{thm:main2} when $k$ is a countable field, but by modifying his construction appropriately we can construct a counterexample.  Thus the following example should be seen as heavily drawing inspiration from the construction of Lorenz.  We suspect, in fact, that the example of Lorenz does not have a prime spectrum that is homeomorphic to that of an algebra that satisfies the Dixmier-Moeglin equivalence.

\begin{example} Let $k=\bar{\mathbb{Q}}$.  Then there exists a finitely generated infinite-dimensional prime noetherian $k$-algebra $R$ such that $R$ does not satisfy the Dixmier-Moeglin equivalence and ${\rm Spec}(R)$ is homeomorphic to ${\rm Spec}(k[t])$.  Thus neither Theorem \ref{thm:main1} nor Theorem \ref{thm:main2} hold when one removes the hypotheses on the cardinality of the base field.
\label{exam:lor}
\end{example}
To do this we require a few basic results.  We begin our construction by taking $B=k[x^{\pm 1}, y^{\pm 1}]$ and letting $\sigma$ be the $k$-algebra automorphism of $B$ given by $x\mapsto x^5 y^4$ and $y\mapsto xy$.  We note that $\sigma$ is an automorphism since it has inverse given by $x\mapsto xy^{-4}$, $y\mapsto x^{-1}y^5$.
\begin{lemma}
\label{lem:unit}
The algebra $B$ has no nonzero proper principal $\sigma$-invariant ideals.  \end{lemma}
\begin{proof} Suppose that $(f(x,y))$ is $\sigma$-invariant with $f\neq 0$.  Then $\sigma(f) = \gamma x^p y^q f$ for some nonzero $\gamma$ and some integers $p$ and $q$.  
We write $f=\sum c_{i,j} x^i y^j$ and since $f$ is a non-unit and nonzero, we have that the set of $(i,j)$ such that $c_{i,j}$ is nonzero has at least two elements.  By multiplying $f$ by a unit we may assume that $c_{0,0}\neq 0$.
Then we have
$$\sigma(f) = \sum c_{i,j} x^{5i +j} y^{4i+j} = \sum c_{i,j} \gamma x^{p+i} y^{q+j}.$$
So now let $M:\mathbb{Z}^2\to \mathbb{Z}^2$ be the $\mathbb{Z}$-linear map given by $M(m,n)=(5m+n,4m+n)$ and let $\Phi:\mathbb{Z}^2\to \mathbb{Z}^2$ be the $\mathbb{Z}$-affine linear map given by
$\Phi(m,n) = (5m+n-p, 4m+n-q)=M(m,n)-(p,q)$.  Let $\mathcal{T}=\{(0,0)=(i_1,j_1),\ldots ,(i_d,j_d)\}$ denote the set of pairs $(i,j)$ for which $c_{i,j}\neq 0$.
 Then by construction the orbit of $(0,0)$ under $\Phi$ must be contained in $\mathcal{T}$ and hence is finite.
 But notice that $\Phi^n(0,0)=(I+M+M^2+\cdots + M^{n-1})(-p,-q),$ which is infinite unless $p=q=0$.  But now this means that the set of values taken by $(M-I)\circ\Phi^n(0,0)-(p,q)= M^n(-p,-q)$ must be finite.  But $M$ is invertible and has eigenvalues that are not roots of unity and so this only occurs if $(p,q)=(0,0)$.  Thus $\Psi=M$.  Now since $f$ is a non-unit there is some $(i,j)\neq (0,0)$ such that $(i,j)\in \mathcal{T}$.  Again, the orbit of $(i,j)$ under $\Psi=M$ must lie in $\mathcal{T}$ and hence must be finite.  But this is impossible since $(i,j)\neq (0,0)$ and $M$ is invertible and has no eigenvalues that are roots of unity.  The result follows.
 \end{proof}
 
Now we take $A=B[u^{\pm 1}]$ and we extend $\sigma$ to $A$ by declaring that $\sigma(u)=2u$.  We now characterize the $\sigma$-prime ideals of $A$.
\begin{lemma} Let $I$ be a nonzero $\sigma$-prime ideal of $A$.  Then $I=JA$, for some $\sigma$-prime ideal $J$ of $B$ having finite codimension in $B$.
\label{lem:A}
\end{lemma}
\begin{proof}
We first claim that $I$ intersects $B$ non-trivially.  To see this, let $a=\sum_{i=0}^m b_i u^i$ be a nonzero element of $I$ with $m$ minimal.  By minimality of $m$ we have $b_0\neq 0$.  We claim that $m=0$.  To see this, observe that $\sigma(a)= \sum_{j=0}^m \sigma(b_i) 2^i u_i$.  So 
$$2^m \sigma(b_m) a - b_m\sigma(a) =\sum_{j=0}^{m-1} (2^m \sigma(b_m)b_j - 2^j b_m \sigma(b_j) )u^j\in I.$$ By minimality of $m$ we see that 
$$2^m \sigma(b_m)b_0 = b_m \sigma(b_0),$$ and so if we set $c=b_0/b_m$ then we have
$\sigma(c)=2^m c$.  Now we write $c=r/s$ with $r,s$ relatively prime elements of $B$.  Then $\sigma(c) = \sigma(r)/\sigma(s)$ and so 
$2^m r \sigma(s) = \sigma(r)s$.  Since $B$ is a UFD and $r$ and $s$ are coprime we see that $\sigma(s)=u_1 s$ and $\sigma(r)=u_2 r$ for some units $u_1, u_2$ of $B$.  But this means that $(s)$ and $(r)$ are nonzero $\sigma$-invariant ideals of $B$ and by Lemma \ref{lem:unit} we then see that $s$ and $r$ are units and so $c$ is a unit in $B$.  But then $c=\gamma x^p y^q$ and it is straightforward to see that $\sigma(c)\neq 2^m c$ has no solutions of this form unless $m=0$ and $\gamma$ is constant.  It follows that $J:=I\cap B$ is nonzero and since $I$ is $\sigma$-prime, $J$ is a $\sigma$-prime ideal of $B$.  Since $J$ is $\sigma$-prime, we have $J=Q\cap \sigma(Q)\cap \cdots \cap \sigma^{n-1}(Q)$ for some nonzero prime ideal $Q$ of $B$ with $\sigma^n(Q)=Q$.  If $Q$ has height one then $Q$ is principal since $B$ is a UFD and hence $J$ is principal.  But this cannot occur by Lemma \ref{lem:unit}.  Thus $Q$ has height at least $2$ and since $B$ has Krull dimension $2$, we see that $Q$ is maximal and so $J$ is of finite codimension in $B$ as required.  Now we claim that $I=JA$. To see this, observe that $QA$ is a height two prime ideal of $A$ and if $I$ is not $JA$ then there must be a $\sigma$-periodic prime ideal $P$ strictly containing $QA$ such that $I$ is the intersection of the $\sigma$-orbit of $P$.  But if $P$ properly contains $QA$ then it must have height at least three and hence must be maximal and so $u-\gamma\in P$ for some nonzero $\gamma$ by the Nullstellensatz.  But $\sigma^n(u-\gamma)=2^n u-\gamma$ and since $\sigma^n(Q)=Q$ for some $n\ge 1$, we then get that 
$(2^n-1)\gamma\in Q$, a contradiction.  
\end{proof}
\begin{proof}[Proof of Example \ref{exam:lor}]
Let $S=A[z^{\pm 1};\sigma]$.  Then we claim that $A$ does not satisfy the Dixmier-Moeglin equivalence and ${\rm Spec}(S)$ and ${\rm Spec}(k[t])$ are homeomorphic.  To see this, first notice that if $\alpha$ and $\beta$ are roots of unity, then the point
 $(\alpha,\beta)$ has finite orbit under $\sigma$ and hence the intersection of the prime ideals in the $\sigma$-orbit of $(x-\alpha,y-\beta)$ is a $\sigma$-prime ideal of $A$.  In particular, since there are infinitely many ordered pairs of roots of unity, we see that there is an infinite set of nonzero $\sigma$-prime ideals of $A$ and by Lemma \ref{lem:A} they are all maximal $\sigma$-prime ideals of $A$.  Notice from the above that if $I$ is a nonzero $\sigma$-prime ideal of $A$ then $A/I\cong (B/Q)[u^{\pm 1}]$ and since $\sigma$ has infinite order on $(B/Q)[u^{\pm 1}]$ and $I$ is a maximal $\sigma$-prime ideal of $A$ we see that $(A/I)[z^{\pm 1};\sigma]$ is simple.  Thus we have shown that $(0)$ is not locally closed and all nonzero prime ideals of $S$ are maximal.  Moreover, since $S$ is noetherian we have that any infinite set of maximal ideals is Zariski dense.  Thus ${\rm Spec}(S)$ is homeomorphic to ${\rm Spec}(k[t])$.  Thus to finish the proof it suffices to show that $(0)$ is rational.  Notice this occurs, if and only if there is a non-constant rational map $f: {\rm Spec}(A)\dashrightarrow \mathbb{P}^1$ such that $f\circ \sigma = f$ \cite[Lemma 3.5]{BG18}.  Notice that we can write 
 $f$ as $P(x,y,u)/Q(x,y,u)$ where $P,Q\in A$ are coprime.  Then $f\circ \sigma=f$ gives that 
 $P(x^5y^4,xy,2u)/Q(x^5y^4,xy,2u) = P(x,y,u)/Q(x,y,u)$ and so 
 $$P(x^5y^4,xy,2u) Q(x,y,u) = P(x,y,u)Q(x^5y^4,xy,2u).$$  Then since $P$ and $Q$ are coprime and $A$ is a UFD, we have
 $(P(x^5y^4,xy,2u))=(P(x,y,u))$ and $(Q(x^5y^4,xy,2u))=(Q(x,y,u))$.  By considering the degree in $u$ and using the fact that $A^* = k^*\langle x^{\pm 1},y^{\pm 1},u^{\pm 1}\rangle$, we see that
 $\sigma(P) = \gamma x^p y^q P$ and $\sigma(Q) = \delta x^s y^t Q$ with $\gamma,\delta\in k^*$ and $p,q,s,t\in \mathbb{Z}$. Now we consider $P$ as a polynomial in $u$ and write
 $P=\sum P_i u^i$.  Then $\sigma(P) = \gamma x^p y^q P$ gives $\sigma(P_i) 2^i = \gamma x^p y^q P_i$ and so each $P_i$ is a $\sigma$-invariant prime ideal of $P$ and hence $P_i$ is either $0$ or a unit in $B$ by Lemma \ref{lem:unit}.
But the units of $B$ are of the form $k^*  x^p y^q$ and it is straightforward to check that there are no unit solutions to 
  $\sigma(P_i) 2^i = \gamma x^p y^q P_i$ if $i>0$.  Hence $P=P_0\in B$ and similarly $Q=Q_0\in B$ and from the above remarks we have that $P_0$ and $Q_0$ are units in $B$ and so $f$ is a unit of $B$.  Thus $f= \gamma x^p y^q$ for some nonzero $\gamma$ and some integers $p$ and $q$.  But now $f\circ\sigma = \gamma x^{5p+q}y^{4p+q} = \gamma x^p y^q$ and so $p=q=0$ and so we see $f$ is necessarily constant and thus $P$ is rational.  Thus $S$ does not satisfy the Dixmier-Moeglin equivalence as claimed.  
 \end{proof}
 \section{Morita Equivalence and corners}
In this section we prove that satisfying the Dixmier-Moeglin equivalence is a Morita invariant. Large parts of this result were already well known and the result itself is not too difficult, but since it doesn't appear in the literature and since there are non-trivial consequences of this result, we find it useful to record this fact. We first make the following well known remark, which shows that Morita equivalence is a much stronger condition than assuming homeomorphic spectra.

\begin{remark}\label{20} Let $R$ and $S$ be Morita equivalent algebras. Then there is a homeomorphism $\Psi: {\rm Spec}(R) \to {\rm Spec}(S)$ with the property that $\Psi(P)$ is primitive if and only if $P$ is primitive.
\end{remark}
\begin{proof} This is essentially given by Theorem 3.5.9 (i) of McConnell and Robson \cite{McR}. Let ${}_SM_R$ be a progenerator. Then the map $\Psi$ from ideals of $R$ to ideals of $S$ given by $I\mapsto MIM^*$ gives a semigroup isomorphism between the ideals of $R$ and those of $S$ that induces a bijection from ${\rm Spec}(R)$ to ${\rm Spec}(S)$ that preserves primitivity \cite[Theorem 3.5.9]{McR}. We note that $\Psi:{\rm Spec}(R)\to{\rm Spec}(S)$ and $\Psi^{-1} :{\rm Spec}(S)\to {\rm Spec}(R)$ are continuous since the collection of prime ideals in $R$ containing $I$ is mapped to the collection of prime ideals in $S$ that contain $\Psi(I)$ and conversely 
$$\Psi^{-1}(\{P\in {\rm Spec}(S)\colon P\supseteq J\}) = \{Q\in {\rm Spec}(R)\colon Q\supseteq \Psi^{-1}(J)\}.$$
\end{proof}

We now show that satisfying the Dixmier-Moeglin equivalence is a Morita invariant.  
\begin{proof}[Proof of Theorem \ref{thm:main3}(a)]
Let ${}_SM_R$ be a progenerator and let $\Psi :{\rm Spec}(R)\to {\rm Spec}(S)$ be the homeomorphism described in Remark \ref{20}.  Let $P$ be a prime ideal of $R$.  Then since being locally closed is a topological property we see that $P$ is locally closed in ${\rm Spec}(R)$ if and only if $\Psi(P)$ is locally closed in ${\rm Spec}(S)$.  By Remark \ref{20} we have that $P$ is primitive if and only if $\Psi(P)$ is primitive.  Finally to see $P$ is rational if and only if $\Psi(P)$ is rational, suppose that $P$ is a rational prime ideal of $R$.  Then by McConnell and Robson \cite[Theorem 3.5.9 (ii)]{McR} we have $R/P$ and $S/\Psi(P)$ are Morita equivalent and hence $Q(R/P)$ and $Q(S/\Psi(P))$ are Morita equivalent \cite[Proposition 3.6.9]{McR}. Thus there is a $k$-algebra isomorphism $Z(Q(R/P))\cong Z(Q(S/\Psi(P))$ (cf. \cite[Theorem 3.59(iii)]{McR}), and so $P$ is rational if and only if $\Psi(P)$ is rational. 
\end{proof}
We note that if $R$ is a ring and $e$ is an idempotent of $R$ then $eRe$ is not in general Morita equivalent to $R$; one typically requires that $e$ be full; i.e., that $ReR=R$.  So in general the Dixmier-Moeglin equivalence being satisfied by $R$ is not equivalent to being satisfied by $eRe$.  For example, if $R$ is a ring that does satisfy the Dixmier-Moeglin equivalence and $S$ does then $T=R\times S$ does not, but $R$ is of the from $eTe$ for an idempotent of $T$.  On the other hand, we are able to show that if $R$ satisfies the Dixmier-Moeglin equivalence then $eRe$ must too.  To do this, we begin with an easy remark, which is folklore, although we are unable to find a reference so we give a proof.
\begin{remark} Let $R$ be a ring with nonzero idempotent $e$.  Let $U$ denote the open subset of ${\rm Spec}(R)$ consisting of prime ideals for which $e\not\in P$.  Then there is a continuous surjection from $U$ (endowed with the subspace topology) to ${\rm Spec}(eRe)$ given by $P\in U\mapsto ePe$.
\label{rem:spec}
\end{remark}
\begin{proof} If $P\in {\rm Spec}(R)$ and $e\not\in P$ then $ePe$ is a proper ideal of $eRe$ and the fact that $P$ is prime in $R$ immediately gives that $ePe$ is prime in $eRe$.  Now to get surjectivity of the map let $Q\in {\rm Spec}(eRe)$ and let $P$ denote the sum of all ideals $I$ such that $eIe\subseteq Q$. Then this is a non-empty sum since $RQR$ has the property that $eRQRe=Q$.  Moreover, $P$ is a proper ideal and $ePe=Q$, since $ePe\subseteq Q$ by construction and $P\supseteq RQR$ so  $ePe\supseteq Q$.  Finally, to check that $P$ is prime, we note that if $J_1$ and $J_2$ are ideals of $R$ with $J_1J_2\subseteq P$.  Then $eJ_1J_2e\subseteq Q$ and so $(eJ_1e)(eJ_2e)\subseteq Q$.  But since $Q$ is prime, we then have either $eJ_1e\subseteq P$ or $eJ_2e\subseteq P$ and so either $J_1$ or $J_2$ is contained in $P$ by definition and so $P$ is prime.  Finally, to see continuity observe that the preimage of the set of prime ideals of $eRe$ that contains an ideal $I$ is precisely the prime ideals of $U$ that contain $RIR$.  
\end{proof}

\begin{proof}[Proof of Theorem \ref{thm:main3}(b)] Let $Q$ be a prime ideal of $eRe$ and let $P$ be a prime ideal of $R$ such that $eRe=P$.  Then we replace $R$ by $R/P$ and assume that $R$ and $eRe$ are prime and that $P$ and $Q$ are zero. First, we have that $(0)$ is a primitive ideal of $R$ if and only if $(0)$ is a primitive ideal of $eRe$ by a result of Lanski, Resco, and Small \cite[Theorem 1]{LRS79}.  Next observe that if $(0)$ is locally closed in ${\rm Spec}(R)$ then the intersection of the nonzero prime ideals of $R$ is a nonzero ideal $I$.  Since every nonzero prime ideal of $eRe$ is of the form $ePe$ for some nonzero prime ideal $P$ of $R$ we see that the intersection of the nonzero prime ideals of $eRe$ contains $eIe$, which is nonzero since $R$ is prime.  Thus $(0)$ is locally closed in ${\rm Spec}(eRe)$ if $(0)$ is locally closed in ${\rm Spec}(R)$.  Conversely, if $(0)$ is not locally closed in ${\rm Spec}(R)$ then $$\bigcap P = (0),$$ where we take the intersection of all nonzero prime ideals $P$ of $R$. 
But since $ePe$ is either a nonzero prime ideal of $eRe$ or is equal to $eRe$ we see that the intersection of the nonzero prime ideals of $eRe$ is contained in the intersection $\bigcap_{P\neq 0} ePe = e(\bigcap_{P\neq 0} P)e=(0)$ and so $(0)$ is not locally closed in ${\rm Spec}(eRe)$.  Finally, $(0)$ is a rational prime of $R$ if and only if the centre of $Q(R)$ is an algebraic extension of $k$.  But now $e$ becomes a full idempotent in $Q(R)$ since the ideal $ReR$ contains a regular element and hence $Q(R)$ and $eQ(R)e$ are Morita equivalent and since $Q(R)$ is prime Artinian, so is $eQ(R)e$ \cite[Proposition 3.5.10]{McR}.  In particular, since $Q(eRe)$ is a localization of $eQ(R)e$ we see that $Q(eRe)=eQ(R)e$ and so $(0)$ is rational in $R$ if and only if $(0)$ is rational in $eRe$ (cf. \cite[Theorem 3.59(iii)]{McR}).  
Finally since every prime ideal of $eRe$ is of the form $ePe$ for some prime ideal of $R$ we see that if $R$ satisfies the Dixmier-Moeglin equivalence then so does $eRe$.
\end{proof}
\begin{remark} As noted earlier, the converse of this result does not hold: if $R$ and $S$ are algebras and $R$ satisfies the Dixmier-Moeglin equivalence and $S$ does not, then $T:=R\times S$ does not satisfy the Dixmier-Moeglin equivalence since $S$ is a homomorphic image.  On the other hand $R=eTe$ with $e=(1,0)$.  The proof in the other direction fails because the argument gives no information about primes $P\in {\rm Spec}(R)$ with $e\in P$.  Thus the converse only holds in general if we know $e$ is full; i.e., $R=ReR$.  But in this case $R$ and $eRe$ are Morita equivalent and the result is already covered by Theorem \ref{thm:main3}(a).
\end{remark}
We now give an application of this result.  We recall that if $A$ is a $k$-algebra and $H$ is a Hopf $k$-algebra, then $H$ \emph{acts on} $A$ if there is there is a $k$-bilinear map $\phi:H\times A\to A$ (where we let $h\cdot a$ denote $\phi(h,a)$) such that for $a,b\in A$ we have $h\cdot (ab) = \sum_i (f_i\cdot a)(g_i\cdot b),$ where $\Delta(h)=\sum f_i \otimes g_i$, and we have $h\cdot 1=\epsilon(h)$.  Given a Hopf algebra $H$ acting on $A$, we can then construct the invariant subalgebra $A^H = \{a\in A\colon h\cdot a=\epsilon(h)a~{\rm for~all}~h\in H\}$.
\begin{corollary}\label{22} Let $k$ be a field, let $A$ be a left noetherian $k$-algebra of finite Gelfand-Kirillov dimension, and let $H$ be a finite-dimensional semisimple Hopf algebra that acts on $A$. If $A$ satisfies the Dixmier-Moeglin equivalence, then the invariant subalgebra $A^H$ satisfies the Dixmier-Moeglin equivalence.
\end{corollary}
\begin{proof} Since $H$ is finite-dimensional, the smash-product algebra $A\#H$ is a finite free left and right $A$-module and hence left noetherian and $A\#H$ satisfies the Dixmier-Moeglin equivalence by \cite{Let89}. Now since $H$ is semisimple, the trace map from $A$ to $A^H$ is surjective and so \cite[Lemma 4.3.4]{mont} gives that $e(A\#H)e\cong A^H$ for some nonzero idempotent $e\in A\#H$. Theorem \ref{thm:main3}(b) now gives that $A^H$ satisfies the Dixmier-Moeglin equivalence.
\end{proof}
\begin{remark} We observe that if one follows the proof then we see that the semisimple hypothesis can be replaced by the weaker condition that the trace map $t:A\to A^H$ being surjective in Corollary \ref{22}.
\end{remark}

\section{Dixmier-Moeglin equivalence and Tensor Products}

In this section, we prove that if $R$ and $S$ are algebras satisfying the Dixmier-Moeglin equivalence then $R\otimes_k S$ does too under a hypothesis on the base field.  We note that there are some subtleties that arise in general since a tensor product of prime rings need not be prime, and so we find obtaining the result without some sort of hypothesis that gives the Nullstellensatz to be difficult.
\begin{proof}[Proof of Theorem \ref{thm:main3}(c)] 
Since $R\otimes_k S$ satisfies the Nullstellensatz, it suffices to prove that rational prime ideals are locally closed. 

Let $P\in {\rm Spec}(R\otimes_k S)$.  Let $Q=\{a\in R \colon a\otimes 1\in P\}$.  Then $Q$ is a prime ideal of $R$ since if 
$a_1,a_2\in R$ are such that $a_1Ra_2\subseteq Q$ then $(a_1\otimes 1)(R\otimes S)(a_2\otimes 1)\subseteq Q\otimes S\subseteq P$ and so $a_1\otimes 1$ or $a_2\otimes 1\in P$, which then gives $a_1$ or $a_2$ is in $Q$.

Then $(R\otimes_k S)/P$ is a homomorphic image of $(R/Q)\otimes S$ and so without loss of generality we may replace $R$ by $R/Q$ and assume that $R$ is prime and that $(R\otimes 1)\cap P=(0)$.  Similarly, we may assume that $S$ is prime and that $(1\otimes S)\cap P=(0)$.  Notice that if $a\otimes b\in P$ with $a\in R$ and $b\in S$ then for $x\in R$ and $y\in S$ we have $ax\otimes yb= (1\otimes y)(a\otimes b)(x\otimes 1)\in P$.  Hence $(a\otimes 1)(R\otimes S)(1\otimes b)\subseteq P$ and thus by the above reduction we have either $a=0$ or $b=0$.  

Now suppose that $P$ is rational.  We show that $P$ is locally closed.  To see this, let $U$ denote the regular elements of $R$ and let $T$ denote the regular elements of $S$.  Then $(U\otimes T)$ is an Ore set of elements of $R\otimes S$ and by the above remarks the prime $P$ survives in the localization $(U\otimes T)^{-1} R\otimes S$. In particular $Q(R)\otimes_k Q(S)$ is a subalgebra of $Q((R\otimes_k S)/P)$ and $Z(Q(R))\otimes_k Z(Q(S))$ is a subalgebra of $Q((R\otimes_k S)/P)$.

In particular, since $P$ is rational, $Z(Q(R))\otimes_k Z(Q(S))$ is an algebraic extension of $k$ and thus $Z(Q(R))$ and $Z(Q(S))$ are algebraic extensions of $k$ and so $(0)$ is a rational prime of $R$ and $(0)$ is a rational prime of $S$.  By hypothesis, we then have that $(0)$ is a locally closed ideal of $R$ and $(0)$ is a locally closed ideal of $S$. 

To see that $P$ is locally closed, we let $X$ denote the set of primes in ${\rm Spec}(R\otimes_k S)$ that properly contain $P$.  We let $X_1$ denote the subset of primes $Q\in X$ such that $Q\cap (R\otimes 1)\neq (0)$ and we let $X_2$ denote the set of primes $Q\in X$ such that $Q\cap (1\otimes S)\neq (0)$.  Arguing as above, we see that for $Q\in X_1$, $Q\cap (R\otimes 1)$ is a nonzero prime ideal of $R$.  Since $(0)$ is locally closed, we then have that there is some nonzero $a\in R$ such that $a\otimes 1 \in \cap_{Q\in X_1} Q$.  Similarly, there is some nonzero $b\in S$ such that $1\otimes b\in \cap_{Q\in X_2} Q$.  

Finally, consider the primes $Q$ such that $Q\cap (R\otimes 1)=(0)$ and $Q\cap (1\otimes S)=(0)$.  Then as before, these primes have trivial intersection with the Ore set $U\otimes T$, and so the primes in $X\setminus (X_1\cup X_2)$ survive in the localization $Q(R)\otimes_k Q(S)/\widetilde{P}$, where $\widetilde{P}= (U\otimes T)^{-1} P$.  We claim that $\widetilde{P}$ is maximal.  To see this, let $I$ be an ideal of $Q(R)\otimes_k Q(S)$ that properly contains $\widetilde{P}$.  Then we pick $x=\sum_{i=1}^d a_i\otimes b_i \in I\setminus \widetilde{P}$ with $a_1,\ldots ,a_d,b_1,\ldots, b_d$ nonzero and $d$ minimal.  Then $d>1$ since $a_1\otimes b_1$ is a unit in $Q(R)\otimes_k Q(S)$.  

Then we may right-multiply by $a_1^{-1}\otimes b_1^{-1}$ and assume that $a_1=b_1=1$.  Then 
for $c\in Q(R)$ we have $[x,c\otimes 1]=\sum_{i=2}^d [a_2,c]\otimes b_2\in I$ and so by minimality of $d$ we see that it is in $\widetilde{P}$; similarly, $[x,1\otimes c']\in \widetilde{P}$ for $c'\in Q(S)$.  Thus $x$ is central mod $\widetilde{P}$.  But
$$Z(Q(R)\otimes Q(S)/\widetilde{P})\subseteq   Z(Q(Q(R)\otimes_k Q(S)/\widetilde{P})) = Z(Q(R\otimes S)/P),$$ which is an algebraic extension of $k$ since $P$ is rational.  Hence $Z(Q(R)\otimes_k Q(S)/\widetilde{P})$ is a finite-dimensional $k$-algebra that is a domain and hence it is a field.  Thus $x$ is a unit, which gives that $I=Q(R)\otimes_k Q(S)$ and so we obtain the desired result.

Thus $X=X_1\cup X_2$ and so we see that $0\neq a\otimes b\in \bigcap_{Q\in X} Q$ and so $P$ is locally closed.  Thus we have obtained the result when ${\rm dim}_k(R),{\rm dim}_k(S)<|k|$.  
\end{proof}
It would be interesting to remove the hypothesis that $R\otimes_k S$ satisfies the Nullstellensatz in the proof of Theorem \ref{thm:main3}(c). We note that a tensor product of noetherian algebras need not be noetherian; for example, if $R=S=k(t_1,t_2,\ldots )$ then $R$ and $S$ are noetherian but $R\otimes_k S$ is not.
 
\begin{remark} Throughout we have taken rationality of a prime ideal $P$ of a left noetherian algebra $R$ to mean that $Z(Q(R/P))$ is an algebraic extension of the base field.  If one instead takes rationality to mean that $Z(Q(R/P))$ is a finite extension of the base field then all results go through, essentially verbatim.
\end{remark}

\end{document}